\documentclass[letterpaper, 10 pt, conference]{ieeeconf}  % Comment this line out if you need a4paper
\IEEEoverridecommandlockouts

\usepackage{graphicx} % for pdf, bitmapped graphics files
\usepackage{amsmath} % assumes amsmath package installed
\usepackage{amssymb}  % assumes amsmath package installed

\usepackage{amsthm}
\usepackage{mathtools}
\usepackage{bm}
\usepackage{algorithm}
\usepackage{cite}
\usepackage{algpseudocode}[noend]

\usepackage{soul}
\usepackage[dvipsnames]{xcolor}

\usepackage[linkcolor=blue,colorlinks=true]{hyperref}

\theoremstyle{definition}
\newtheorem{definition}{Definition}
\newtheorem{theorem}{Theorem}
\newtheorem{lemma}{Lemma}

\newtheorem{remark}{Remark}

\newtheorem{assumption}{Assumption}

\usepackage{subcaption}

\newtheorem{objective}{Objective}

\usepackage{caption}
\DeclareCaptionType{equ}[][]

\graphicspath{{cdc2020figures/}} %Setting the graphicspath

\title{\LARGE \bf
Trust-based Rate-Tunable Control Barrier Functions for Non-Cooperative Multi-Agent Systems
}
\author{Hardik Parwana$^{1}$, Aquib Mustafa$^{2}$, and Dimitra Panagou$^{3}$ \\
\textit{This work has been submitted for review to a conference}
\thanks{$^{1}$Hardik Parwana is with Robotics Institute, University of Michigan, Ann Arbor, MI 48109, USA, {\tt\small hardiksp@umich.edu}}
\thanks{$^2$Aquib Mustafa is with Advanced Technologies and Research, Sensia LLC (Rockwell Automation + Schlumberger), Houston, TX 77079, USA, {\tt\small aquibmustafa.225@gmail.com} }
\thanks{$^3$Dimitra Panagou is with Department of Aerospace Engineering and Robotics Institute, University of Michigan, Ann Arbor MI 48109, {\tt\small dpanagou@umich.edu}}
}
\newcommand{\reals}{\mathbb{R}}
\newcommand{\s}{\mathcal{S}}
\newcommand{\X}{\mathcal{X}}
\newcommand{\U}{\mathcal{U}}

\newcommand{\K}{\mathcal{K}}

\newcommand{\C}{\mathcal{C}}

\newcommand{\classK}{\mbox{class-$\K$} }

\newcommand{\Int}{\text{Int} }

\newcommand{\eqn}[1]{\begin{align}#1\end{align}}

\begin{document}

\maketitle
\thispagestyle{empty}
\pagestyle{empty}

%%%%%%%%%%%%%%%%%%%%%%%%%%%%%%%%%%%%%%%%%%%%%%%%%%%%%%%%%%%%%%%%%%%%%%%%%%%%%%%%

\begin{abstract}
For efficient and robust task accomplishment in multi-agent systems, an agent must be able to distinguish cooperative agents from non-cooperative agents, i.e., uncooperative and adversarial agents. Task descriptions capturing safety and collaboration can often be encoded as Control Barrier Functions (CBFs). In this work, we first develop a trust metric that each agent uses to form its own belief of how cooperative other agents are. The metric is used to adjust the rate at which the CBFs allow the system trajectories to approach the boundaries of the safe region. Then, based on the presented notion of trust, we propose a Rate-Tunable CBF framework that leads to less conservative performance compared to an identity-agnostic implementation, where cooperative and non-cooperative agents are treated similarly. Finally, in presence of non-cooperating agents, we show the application of our control algorithm to heterogeneous multi-agent system through simulations.

\end{abstract}

\section{Introduction}

Collaborating robot teams can enable tasks such as payload transportation, surveillance and exploration \cite{rizk2019cooperative}. The ability to move through an environment under spatial and temporal constraints such as connectivity maintenance, collision avoidance, and waypoint navigation is essential to successful operation. In principle, the group of agents is expected to complete desired tasks (in terms of goal-reaching and safety specifications) even in the presence of non-cooperative agents. With safety specifications encoded as Control Barrier Functions (CBFs), this work develops a trust metric that each agent uses to distinguish cooperative from non-cooperative agents(that includes uncooperative and adversarial agents as defined formally in Section \ref{section::agent_types}). This trust metric is then used to shape the response of the controller by adjusting the CBF parameters. As a motivational example, consider the scenario depicted in Fig.\ref{fig::scenario1}, where a green robot will be intercepted by a red robot if it follows a nominal trajectory of moving forward. The red robot wishes to cause no harm, but gives priority to its own task. Therefore, if the green robot has to believe, based on previous observations, that the red robot will not chase it, then it can adjust its response and minimize deviation from its nominal trajectory.

\begin{figure}[htp]
    \centering
    \includegraphics[scale=0.4]{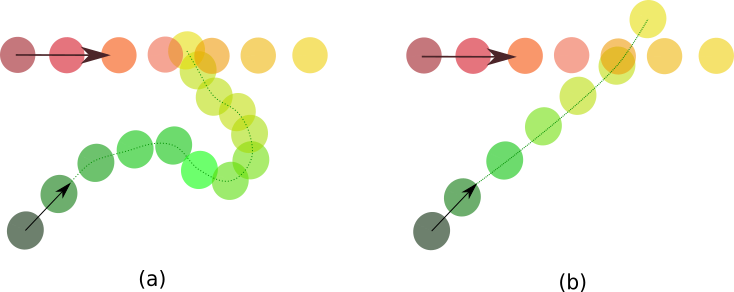}
    \caption{\small{Green robot moving closer to a red uncooperative robot. Gradations to lighter colors show progress in time.
    %Black arrows show the velocity direction and magnitude along the vertical axis. 
    A fixed value of $\alpha$, a controller parameter in (\ref{eq::cbf_condition}), might force the group to back off, as in (a), and deviate significantly from nominal trajectory. Whereas, a trust based relaxation would adjust $\alpha$ to allow a more efficient motion in (b) where it slows down enough to let red robot pass first and then continue on its motion in nominal direction.}}
    \label{fig::scenario1}
    \vspace{-2mm}
\end{figure}
%Most of the existing research dealing with adversarial or uncertain elements in a system adopts worst-case analysis which can be too conservative. First, a trust factor, which represents the belief of other agents being cooperative, uncooperative, or adversarial, is designed based on expected motions of other agents. Then, it is used to relax or tighten the constraints. By determining the contribution to safety solely through an adaptive CBF and its derivatives, the algorithm is also inherently agnostic to dynamics and  thus, enables implementation on heterogeneous systems. 

%This abstraction of high-level tasks also makes our algorithm agnostic to dynamics and controllers employed, thus ensuring safety among heterogeneous agents.
Specifying tasks for multi-agent systems and designing safe controllers for successful execution has been an active research topic  \cite{lindemann2019control,lindemann2020barrier}. While research on multi-agent systems with non-cooperative agents does not provide all these guarantees, several remarkable results still exist on resilient control synthesis \cite{pasqualetti2011consensus,Sundaram2011,usevitch2019resilient,saulnier2017resilient,zhou2018resilient,guerrero2019realization,mustafa2019attack,pirani2019design}. \cite{saulnier2017resilient} and \cite{zhou2018resilient}, the authors present resilient algorithms for flocking and active target-tracking applications, respectively. In \cite{pirani2019design} and \cite{mustafa2019attack}, the authors design game-theoretic and adaptive control mechanisms to directly reject the effect of adversaries without the need of knowing or detecting their identity. Recently, CBF-based approaches have also been presented for achieving resilience under safety and goal-reaching objectives. In particular,  \cite{borrmann2015control} developed multi-agent CBFs for non-cooperative agents.
%cooperative, uncooperative, and adversarial robots,
%and \cite{cheng2020safe} extended it to a scenario with uncertain dynamics wherein state-dependent disturbances are learned using Gaussian processes. 
\cite{usevitch2021adversarial} designed a mechanism to counter collision-seeking adversaries by designing a controller based on CBFs that is robust to worst-case actions by adversaries. All of the above studies either assume that the identities of adversarial agents is known \textit{a priori}, or they design a robust response without knowing or detecting their identities, both of which can lead to conservative responses. Moreover, most of these works also assume that each agent knows the exact dynamics of the other agents. 
%and their dynamics is known apriori or design a response resilient to identities, both of which can lead to conservative responses.
To the best of our knowledge, no prior studies have employed CBFs to make inferences from the behavior of surrounding agents, and tune their response to reduce conservatism while still ensuring safety.

In this paper, we make two contributions to mitigate the above limitations. First, we introduce the notion of trust, and propose a trust model that each agent uses to develop its own belief of how {cooperative} other agents are. Since CBFs ensure safety by restricting the rate of change of barrier functions along the system trajectories \cite{ames2016control}, we design the trust metric based on how robustly these restrictions are satisfied. Each agent rates the behavior of a neighbor agent on a continuous scale, as either 1) being cooperative, i.e., an agent that actively gives priority to safety, or 2) being uncooperative, i.e., an agent that does not actively try to collide, however, gives priority to its own tasks, and thus disregards safety or 3) being adversarial, i.e., an agent that actively tries to collide with other agents. Second, we design an algorithm that tightens or relaxes a CBF constraint based on this trust metric, and develop Trust-based, Rate-Tunable CBFs (trt CBFs) that allow {online} modification of the parameter that defines the class-$\mathcal{K}$ function employed with CBFs\cite{ames2016control}. 
Our notion of trust shares inspiration with works that aim to develop beliefs on the behavior of other agents and then use it for path planning and control. Deep Reinforcement Learning (RL) has been used to learn the underlying interactions in multi-agent scenarios \cite{brito2021learning}, specifically for autonomous car navigation, to output a suitable velocity reference to be followed for maneuvers such as lane changing or merging. 
%\red{\cite{cai2021safe} uses Multi-agent RL with outputs filtered by the composition of cooperative and uncooperative CBFs to learn a safe policy.} 
Few works also explicitly model trust for a multi-agent system \cite{valtazanos2011intent, fooladi2020bayesian, hu2021trust}. In particular, \cite{valtazanos2011intent} develops an intent filter for soccer-playing robots by comparing the actual movements to predefined intent templates, and gradually updating beliefs over them. Among the works focusing on human-robot collaboration,  \cite{fooladi2020bayesian} learns a Bayesian Network to predict the trust level by training on records of each robot's performance level and human intervention. \cite{hu2021trust} defines trust by modeling a driver's perception of vehicle performance in terms of states such as relative distance and velocity to other vehicles, and then using a CBF to maintain the trust above the desired threshold. The aforementioned works show the importance of making inferences from observations, however they were either developed specifically for semi-autonomous and collaborative human-robot systems, making use of offline datasets, or when they do consider trust, it is mostly computed with a \textit{model-free} approach, such as RL. {The above works also use the computed trust as} a monitoring mechanism rather than being related directly to the low-level controllers that decide the final control input of the agent and ensure safety-critical operation. In this work, we pursue a \textit{model-based} design of trust metric that actively guides the low-level controller in relaxing or tightening the CBF constraints and help ensure successful task completion.

% BELOW not such relavent reference
%\cite{saeidi2018incorporating} also   trust for interacting modeling human's trust of robots in context of multi-robot syste 

\section{Preliminaries}

\subsubsection{Notations}
The set of real numbers is denoted as $\reals$ and the non-negative real numbers as $\reals^+$. Given $x\in \reals$, $y \in \reals^{n_i}$, and $z\in \reals^{n_i\times m_i}$, $|x|$ denotes the absolute value of $x$ and $||y||$ denotes $L_2$ norm of $y$. The interior and boundary of a set $\C$ are denoted by $\textrm{Int}(\C)$ and $\partial \C$. For $a\in \reals^+$, a continuous function $\alpha:[0,a)\rightarrow[0,\infty)$ is a $\classK$ function if it is strictly increasing and $\alpha(0)=0$. Furthermore, if $a=\infty$ and $\lim_{r\rightarrow \infty} \alpha(r)=\infty$, then it is called class-$\mathcal{K}_\infty$. The angle between any two vectors $x_1,x_2\in \reals^{n_i}, \forall i=1,2$ is defined with respect to the inner product as $\theta = x_1^Tx_2/||x_1||~ ||x_2||$. Also, $\delta y$ represents a infinitesimal variation in $y$.

\subsubsection{Safety Sets}
Consider a nonlinear dynamical system
\eqn{
\dot{x} = f(x) + g(x)u,
\label{eq::dynamics_general}
}
where $x\in \X \subset \reals^{n_i}$ and $u \in \U \subset \reals^{m_i}$ represent the state and control input, and  $f:\X\rightarrow \reals^{n_i}, g:\X\rightarrow \reals^{n_i\times m_i}$ are locally Lipschitz continuous functions. A safe set $\s$ of allowable states be defined as the 0-superlevel set of a continuously differentiable function $h(x):\mathcal{X} \rightarrow \reals$ as follows
\eqn{
        \s & \triangleq \{ x \in \X : h(x) \geq 0 \}, \\
        \partial \s & \triangleq \{ x\in \X: h(x)=0 \}, \\
        \Int (\s) & \triangleq \{ x \in \X: h(x)>0  \}.
\label{eq::safeset}
}

\noindent

\begin{definition}
 (CBF\cite{ames2019control}) Let $\s\subset \mathcal{X}$ be the superlevel set of a continuously differentiable function $h:\mathcal{X}\rightarrow \reals$. $h$ is a CBF on {$\mathcal{X}$} if there exists an extended class-$\mathcal{K}_\infty$ function $\nu$ such that for system \eqref{eq::dynamics_general}
 \eqn{
 \sup_{u\in \mathcal{U}} \left[ \frac{\partial h}{\partial x}( f(x) + g(x)u ) \right] \geq -\nu(h(x)), ~ \forall x \in \mathcal{X}. 
 \label{eq::cbf_derivative}
 }
\end{definition}

\begin{lemma}(\cite[Theorem 2]{ames2019control})
 Let $\s\subset \mathcal{X}$ be the superlevel set of a smooth function $h:\mathcal{X}\rightarrow \reals$. If $h$ is a CBF {on $\mathcal{X}$}, and $\partial h/\partial x\neq 0 ~ \forall x\in \partial \s$, then any Lipschitz continuous controller belonging to the set
 \eqn{
 K(x) = \{ u\in \mathcal{U} : \frac{\partial h}{\partial x}( f(x) + g(x)u ) + \nu(h(x)) \geq 0 \},
 }
for the system \eqref{eq::dynamics_general} renders the set $\s$ safe.
\end{lemma}

In this paper, we consider only the linear class-$\mathcal{K}_\infty$ functions of the form \begin{align}
    \nu(h(x)) = \alpha h(x), \quad \alpha \in \reals^+.
\end{align} 
With a slight departure from the notion of CBF defined above, $\alpha$ is treated as a parameter in this work and its value is adjusted depending on the trust that an agent has built on other agents based on their behaviors. Since $\alpha$ represents the maximum rate at which the state trajectories $x(t)$ are allowed to approach the boundary $\partial \mathcal S$ of the safe set $\mathcal S$, a higher value of $\alpha$ corresponds to relaxation of the constraint, {which allows agents to get closer to each other}, whereas a smaller value tightens it. The design of the trust metric is presented in Section \ref{section::methodology}.

\section{Problem Statement}
 In this paper, we consider a system consisting of $N$ agents with $x_i\in \X_i \subset \reals^{n_i}$ and $u_i\in \U_i \subset \reals^{m_i}$ representing the state and control input of the agent $i\in\mathcal{V}= \{1,2,..,N\}$. The dynamics of each agent is represented as
\eqn{
\dot{x}_i = f_i(x_i) + g_i(x_i)u_i,
\label{eq::dynamics}
}
where $f_i:\X\rightarrow \reals^{n_i}, g_i:\X\rightarrow \reals^{n_i\times m_i}$ are Lipschitz continuous functions. %\red{with Lipschitz constants $b_f \in \mathbb{R_{+}}$ and $b_g\in \mathbb{R_{+}}$, $\forall i \in \mathcal{V}$, respectively.}

The dynamics and the control input of an agent $i$ is not known to any other agent $j\in\mathcal{V} \setminus i $. However, we have the following assumption on the available observations for each agent.
\begin{assumption}
   Let the combined state of all agents be $x = [x_1^T ~ , ..., x_n^T]^T$. Each agent $i$ has perfect measurements $z$ of the state vector $x$. Furthermore, $\dot x_i$ is a Lipschitz continuous function of $z$, i.e., $\dot{x}_i=F_i(z), \forall i\in \mathcal{V}$.
    \label{assumption::true_estimate}
\end{assumption}
The above assumption implies that all agents, including adversarial and uncooperative agents, design their control inputs based on $z$. {Therefore, we assume a fully connected (complete) communication topology. However, each agent computes its own control input independently based on its CBFs that can be different from CBFs of its neighbors($h_{ij}\neq h_{ji}, \alpha_{ij}\neq \alpha_{ji}$ in Section \ref{section::task_specification}).} We also have the following assumption on the estimate that agent $i$ has regarding the closed-loop dynamics $F_j$ of agent $j$.
\begin{assumption}
Each agent $i$ has an estimate $\hat{F}_j$ of the true closed-loop response $F_j(z)$ of other agents $j\in \mathcal{V}\setminus i$ such that there exists $b_{F_j}: \reals^{n_j}\rightarrow \reals^+$ for which
\eqn{
\dot{x}_j \in \hat{F}_j = \{ ~y\in \reals^{n_j}~~ |~~ || F_j(z)-y || \leq b_{F_j}(z) ~ \}
}
\label{assumption::estimate}
\end{assumption}
\vspace{-6mm}
\begin{remark}
Assumption \ref{assumption::estimate} can easily be realized for systems that exhibit smooth enough motions for a learning algorithm to train on. For example, a Gaussian Process \cite{srinivas2012information}, that can model arbitrary Lipschitz continuous functions to  {return estimates with mean and uncertainty bounds}, can use past observations to learn the relationship between $\dot{x}_i$ and $z$. Another way would be to bound $\dot{x}_i$ at time $t+1$ with the value of $\dot{x}_i$ at time $t$ and known, possibly conservative, Lipschitz bounds of the function $F_i$.
\end{remark}

\subsubsection{Task Specification}
\label{section::task_specification}
Agents for which we aim to design a controller are called \textit{intact agents}, and they are cooperative in nature (defined in Section \ref{section::agent_types} ). Each intact agent $i$ is 
assigned a goal-reaching task that it needs to accomplish while maintaining inter-agent safety with other agents. To encode goal reaching task, we design a Control Lyapunov Function (CLF) $V_i : \mathbb R^n \rightarrow \mathbb R$ of the form $V_i(x_i,x_i^r) = || x_i-x_i^r ||^2$, where $x_i^r \in \mathbb R^n $ is the reference state. The safety constraints w.r.t $j\in \mathcal{V}\setminus i$ is encoded in terms of CBFs $h_{ij}$. %Moreover, a continuously differentiable functions $h_{ij}:\mathbb R^n \rightarrow \mathbb R$ acts as barrier functions and employed to encode safety constraints. 
Agent $i$ {ensures that $h_{ij}\geq 0$ is maintained by} imposing the following CBF condition {on $x_i,x_j\in\mathcal{X}$}
\eqn{
\dot{h}_{ij} = \frac{\partial h_{ij}}{\partial x_i}\dot{x}_i + \frac{\partial h_{ij}}{\partial x_j}\dot{x}_j \geq -\alpha_{ij} h_{ij}, 
\label{eq::cbf_condition}
}
where $\alpha_{ij}\in \reals^+$. {Suppose the 0-superlevel set of $h_{ij}$ is given by $\mathcal{S}_{ij}$. Then from Lemma 1, if the initial states $x_i(0),x_j(0)\in \s_{ij}$, then $\s_{ij}$ if forward invariant.} 
%We use exponential CLF for simplicity but Finite-time CLFs[], Fixed-time CLFs[], and time varying barrier functions[] may also be used. We impose satisfaction of $\phi$ as a loose constraint and give priority to satisfaction of predicates $\pi$ at all times.
Note that compared to some works, we do not merge barrier functions into a single function. This allows us to monitor contribution of each neighbor to inter-agent safety individually and design trust metric as detailed in Section\ref{section::trust_metric}.
\subsubsection{Types of Agents}
\label{section::agent_types}
The considered agent behaviors are classified into three types based on their interaction with others. 
\begin{definition}
\textbf{(Cooperative agent)} An agent $j$ is called cooperative to agent $i$ if, given $\dot{x}_i$, agent $j$ always chooses its control input $u_j$ such that \eqref{eq::cbf_condition} holds for some $\alpha_{ij}>0$.
\end{definition}

\begin{definition}
\textbf{(Adversarial Agent)}
An agent $j$ is called adversarial to agent $i$ if {agent $j$ designs its} control input such that following holds
\eqn{
\frac{\partial h_{ij}}{\partial x_j} \dot{x}_j \leq 0 ~~ \forall x_i\in \reals^{n_i},\,x_j\in \reals^{n_j}.
}
\end{definition}
\begin{definition}
\textbf{(Uncooperative agent)}
An agent $j$ is called uncooperative to agent $i$ if its control design disregards any interaction with $i$. That is, $\dot{x}_j \equiv \dot{x}_j(x_j)$ and not $\dot{x}_j \equiv \dot{x}_j(x_i, x_j)$.

\end{definition}

While all the agents are assumed to behave in one of the above ways, the complex nature of interactions in a multi-agent system makes it hard to distinguish between the different types of agents. Moreover, it is not necessary that $h_{ij}=h_{ji}$ and $\alpha_{ij}=\alpha_{ji}$, for example agents may employ different safety radius for collision avoidance. Thus, an agent $i$ may perceive a cooperative agent $j$ to be uncooperative if $j$ cannot satisfy $i$'s prescribed level of safety, given by $\alpha_{ij}$, under the influence of a truly uncooperative or adversarial agent $k$. Therefore, rather than making a fine distinction, this work (1) first, presents a trust metric that rates other agents on a continuous scale as being cooperative, adversarial or uncooperative and then, (2) by leveraging it, an adaptive form of $\alpha_{ij}$ in \eqref{eq::cbf_condition} is presented that adjusts the safety condition based on the $j$'s behavior.
 The effect of poorly chosen parameter $\alpha$ was illustrated in Fig.\ref{fig::scenario1}(a) where a conservative response is seen.
% sensitive to factors such as initial condition, number of neighbors and the speed of uncooperative agents. 
%Consider the scenario depicted in Fig.\ref{fig::scenario1}, where a group of robots will be intercepted by another robot if they were to keep moving forward and follow a  nominal trajectory. This outlier robot wishes to cause no harm, but gives priority to its own task. 
%In Fig.\ref{fig::scenario1}, a fixed value of $\alpha$ would force the group to back-off or move towards the right to evade possible collision. 
However, the green robot can relax the constraint, i.e., increase $\alpha$, to allow itself to get closer to red robot and just slow down rather than turning away from it.
%We now formalize this idea in the following subsection. 
%by empowering intact agents to adapt their parameter $\alpha$ based on their belief of whether agent $j$ is contributing to their safety and if it's. %This rating is dependent on three quantities

\subsubsection{Objective}
In this subsection, we first presented some definitions and then formulate the objective of this paper. 
\begin{definition}
\textbf{(Nominal Direction)} A nominal motion direction $\dot{\hat{x}}_i^n$ for agent $i$ is defined with respect to its global task and Lyapunov function $V_i$ as follows:
\eqn{
\dot{\hat{x}}_i^n = -\frac{\partial V_i}{\partial x_i}/|| \frac{\partial V_i}{\partial x_i} ||.
\label{eq::nominaldirection}
}
\end{definition}
\begin{definition}
\textbf{(Nominal Trajectory)} A trajectory $x_i^n(t)$is called nominal trajectory for an agent $i$, if agent follows the dynamics in \eqref{eq::dynamics} with  some $x_i(0) \in \mathcal{X} $ and control action $u_i=-\gamma\frac{\partial V_i}{\partial x_i}/|| \frac{\partial V_i}{\partial x_i} ||$, where  $\gamma \in \mathbb R$ denotes some design gain.
%\textbf{THE DEFINITION IS NOT RIGOROUS. DO YOU MEAN THE SOLUTION $x_i^(t)$ to the equation 18 from some initial condition $x_i^(0)$? If so, just mention it below (18). In fact, (18) defines some form of nominal gradient flow for the system dynamics.}
%resulting from moving continuously along nominal directions is called nominal trajectory.
\end{definition}

We now define the objective of this paper. The problem is as follows.

\begin{objective}
    Consider a multi-agent system of $N$ agents governed by the dynamics given by \eqref{eq::dynamics_general}, with the set of agents denoted as $\mathcal V$, $|\mathcal V| = N$, the set of intact agents denoted as $\mathcal{N}$, the set of adversarial and uncooperative agents denoted as $\mathcal{A}\neq \emptyset$, with $\mathcal{V}=\mathcal{N}\cup \mathcal{A}$. Assume that the identity of the adversarial and uncooperative agents is unknown. Each agent $i$ has a global task encoded via a Lyapunov function $V_i$, and $N$ local tasks encoded as barrier functions $h_{ij}, j\in \{1,2,..,N\}$. Design a decentralized controller to be implemented by each intact agent $i\in \mathcal{N}$ such that 
    \begin{enumerate}
        \item $h_{ij}(t)\geq 0, ~~ \forall t\geq 0$ for all $i \in \mathcal{N}, j \in \mathcal V$.
        \item The deviation between actual and  nominal trajectory is minimized, i.e., for the designed controller $||x_i(t)-x^n_i(t)||$ is minimum where $x_i^n(t)$ denotes nominal trajectory as defined in Definition 6.
        
        %\textbf{WRITE IN MATH:} 
        %\item \textcolor{}{$(x_i,t) \models \pi_i ~~\forall t>0$ even in presence of adversarial and uncooperative agents.}
        %\item \textcolor{blue}{$(x_i,t) \models \phi$ with minimal deviation of trajectory $x_i(t)$ from its nominal trajectory $x_i^{n}(t)$.}
    \end{enumerate}
\end{objective}

% \begin{assumption}
% The nominal trajectories of all the agents satisfy the following conditions
% \begin{enumerate}
%     \item $x_{d_i}(t) \neq x_{d_j}(t) ~\forall i\neq j$ and $\forall t>=0$. 
%     % \item $x_{d_i}$
% \end{enumerate}
% \end{assumption}

\section{Methodology}
\label{section::methodology}
The intact agents compute their control input in two steps. The first step computes a reference control input $u^r_i$ for an agent $i$ that makes agent to converge to its nominal trajectory. Let the nominal trajectory be given by $x^r_i(t)$ and the goal reaching can be encoded as Lyapunov function $V_i = ||x_i-x^r_i||^2$. Then for $k>0$, a reference control input is designed using exponentially stabilizing CLF conditions as follows\cite{khalil2002nonlinear}
\vspace{-0.2cm}
\begin{subequations}\label{eq::CLF}
    \begin{align}
    \begin{split}\label{eq::CLFa} 
     u^r_i(x) = \arg \min_{u \in \mathcal{U}_i} \quad & u^T u
     \vspace{-0.1cm}
    \end{split}\\ 
    \begin{split} \label{eq::CLFb}
    \textrm{s.t.} \quad & \dot{V}(x_i,u_i) \leq -kV(x_i,x^r_i).
    \end{split}
    \end{align}
% \vspace{-0.1cm}
\end{subequations}
In the second step, the following CBF-QP is defined to minimally modify $u^r_i$ while satisfying {(\ref{eq::cbf_condition}) for all agents $j\in\mathcal{V}\setminus i$ }
\vspace{-0.2cm}
\begin{subequations}\label{eq::CBF}
\begin{align}
\begin{split}\label{eq::CBFa} 
u_i(x,\alpha_{ij}) = \arg \min_{u} \quad & ||u-u^r_i||^2
 \vspace{-0.1cm}
\end{split}\\ 
\begin{split}  \label{eq::CBFb}
\textrm{s.t.} \quad & \frac{\partial h_{ij}}{\partial x_i}\dot{x}_i + \min_{\hat{\dot{x}}_j\in \hat{F}_j}\left\{ \frac{\partial h_{ij}}{\partial x_j}\hat{\dot{x}}_j\right\} \\
  & \quad \quad \geq  -\alpha_{ij} h_{ij},
\end{split}
%\begin{split} \label{eq::CBFb}
%\hrulefill \forall j\in \mathcal{V}\setminus i,
% \vspace{-0.1cm}
%\end{split}
\end{align}
\label{eq::CBF_QP}
\end{subequations}\noindent
where $\alpha_{ij}\in \reals^+$, and $\hat{\dot{x}}_j$ denotes all possible movements of agent $j$ per Assumption \ref{assumption::estimate}. Eq.(\ref{eq::CBF}) provides safety assurance for {$i$ w.r.t the worst-case predicted motion of $j$}. Henceforth, whenever we refer to the \textit{actual motion} of any agent $j$ with respect to $i$, we will be referring to its uncertain estimate that minimizes the contribution to safety,
\eqn{
    \dot{a}_j^i = \arg \min_{\hat{\dot{x}}_j\in \hat{F}_j}\left\{ \frac{\partial h_{ij}}{\partial x_j}\hat{\dot{x}}_j\right\}
    \label{eq::actual_motion}.
}

In the following sections, we first propose Tunable-CBF as a method that allows us to relax or tighten CBF constraints while still ensuring safety. Then, we design the trust metric to adjust $\alpha_{ij}$ and tune the controller response. 

\subsection{Rate-Tunable Control Barrier Functions}
In contrast to the standard CBF where $\alpha$ in (\ref{eq::cbf_condition}) is fixed, we aim to design a method to adapt the values of $\alpha$ depending on the trust metric that will be designed in Section \ref{section::trust_metric}. In this section, we show that treating $\alpha$ as a state with Lipschitz continuous dynamics can still ensure forward invariance of the safe set. 
\begin{definition}(Rate-Tunable CBF)
Consider the system dynamics in \eqref{eq::dynamics_general}, augmented with the state $\alpha\in \reals$ that obeys the dynamics
\eqn{
\begin{bmatrix} \dot x \\ \dot \alpha \end{bmatrix} = \begin{bmatrix} f(x) + g(x) u \\ f'(x, \alpha) \end{bmatrix},
\label{eq::augmented}
}
where $f'$ is a locally Lipschitz continuous function w.r.t $x,\alpha$. Let $\s$ be the safe set defined by a continuously differentiable function $h$ as in~\eqref{eq::safeset}. $h$ is a Rate Tunable-CBF for the augmented system~\eqref{eq::augmented} on $\s$ if
 \eqn{
 \sup_{u\in \reals^{m_i}} \left[ \frac{\partial h}{\partial x}( f(x) + g(x)u ) \right] \geq -\alpha h(x) ~ \forall x \in \mathcal{X} , \alpha\in \reals.
 \label{eq::cbf_derivative}
 }
\end{definition}
Note that the above definition assumes unbounded control input. For bounded control inputs, the condition of existence of class-$\mathcal{K}$ function in standard CBFs would possibly have to be replaced with finding a domain $\mathcal{D}\subset \reals$ over which the state $\alpha$ must evolve so that the CBF condition holds. Such an analysis will be addressed in future.

\begin{theorem}
\label{theorem::tunable_CBF}
Consider the augmented system~(\ref{eq::augmented}), and a safe set $\s$ in \eqref{eq::safeset} be defined by a Rate-Tunable CBF $h$. For a Lipschitz continuous reference controller $u^r: \mathcal{X} \to \reals^{m_i}$, let the controller $u = \pi(x, \alpha)$, where $\pi : \mathcal{X} \times \reals \to \mathcal{U}$ be formulated as
\begin{subequations}\label{eq::theorem_qp}
\begin{align}
\begin{split}\label{eq::theorem_qpa} 
 \pi(x, \alpha) = \underset{u\in \reals^{m_i}}{\operatorname{argmin}} \quad & || u - u^r(x, \alpha) ||^2 
\end{split}\\ 
\begin{split} %\label{eq::theorem_qpb}
\textrm{s.t.} \quad & \frac{\partial h}{\partial x} ( f(x)+g(x)u )  )\geq -\alpha h(x) \label{eq::cbf_alpha}
\end{split}
\end{align}

\end{subequations}
\noindent Then, the set $\s$ is forward invariant. 
%1) the solution, u(x,$\alpha$) is locally Lipschitz continuous for $x \in int(\s)$.  2) 
\end{theorem}

\begin{proof}
The proof involves two steps. First, we show that $\pi(x, \alpha)$ is Lipschitz continuous function of $x,\alpha$. Second, we show that $\s$ is forward invariant, using Nagumo's Theorem\cite[Thm~3.1]{blanchini1999set}. \\
\textit{Step 1}:  
%Since $\frac{\partial h}{\partial x_i}g(x_i)\neq 0$, the following control input
% \eqn{
% u(x) = \left\{ - \left[ \frac{\partial h}{\partial x} g(x) \right]^T\left( \frac{\partial h}{\partial x} f(x)  + \alpha h(x) \right) \right\} / \bigg|\bigg|\frac{\partial h}{\partial x}g(x)\bigg|\bigg|^2
% }
% satisfies (\ref{eq::cbf_alpha}) with equality. Hence, a solution always exists. Since the objective function is quadratic, the solution is also unique[]. 
Since $h$ is a Rate-Tunable CBF, according to~\eqref{eq::cbf_derivative}, there always exists a $u \in \reals^{m_i}$ that satisfies the constraint of the QP~\eqref{eq::cbf_alpha}. Since the QP has a single constraint, the conditions of~\cite[Theorem~3.1]{hager1979lipschitz} are satisfied, and hence the function $\pi$ is Lipschitz with respect to the quantities $h(x), L_fh(x), L_gh(x), u^r(x)$ and $\alpha$.  Since these quantities are Lipschitz wrt to $x$ and $\alpha$, by the composition of Lipschitz functions, $\pi(x, \alpha)$ is Lipschitz with respect to both arguments.
% The data matrices are obtained by comparing the QP to its standard form
% \begin{align}
% \begin{split}\label{eq::theorem_qpa} 
%  u(t) = \arg \min_{u\in \reals^m} \quad & u^T R x + r^T x
% \end{split}\\ 
% \begin{split} \label{eq::theorem_qpb}
% \textrm{s.t.} \quad & A x + b \leq 0
% \end{split}
% \end{align}
% \noindent
% where $A = \frac{\partial h}{\partial x}g(x), b=\frac{\partial h}{\partial x}f(x) + \alpha h(x),R = \mathcal{I},r=-2u_r^T$ with $\mathcal{I}$ being identity matrix. 
% Since $h$ is a continuously differentiable function, $\partial h/\partial x$ and hence all of $A,b,R,r$ are locally Lipschitz continuous functions of $x,u_r,\alpha$. Therefore, $u$ is locally Lipschitz in $x,\alpha$. \\

\textit{Step 2:} Since $f,g,f'$ are Lipschitz functions, the closed-loop dynamics of augmented system \eqref{eq::dynamics_general}, (\ref{eq::augmented}) is also Lipschitz continuous in $x,\alpha$. From \cite[Thm~3.1]{khalil2002nonlinear}, for any $x(0)=x_0\in \mathcal{X}$, there exists a unique solution for all $t \geq 0$. Since at any $x \in \partial \s$, the constraint in the QP forces $\dot h(x, \pi(x, \alpha)) \geq 0$, by Nagumo Theorem's~\cite{blanchini1999set} the $\s$ if forward invariant.
\end{proof}

\vspace{-4.1mm}
\subsection{Design of Trust Metric}
\label{section::trust_metric}
\begin{figure}[htp]
    \centering
    \includegraphics[scale=0.3]{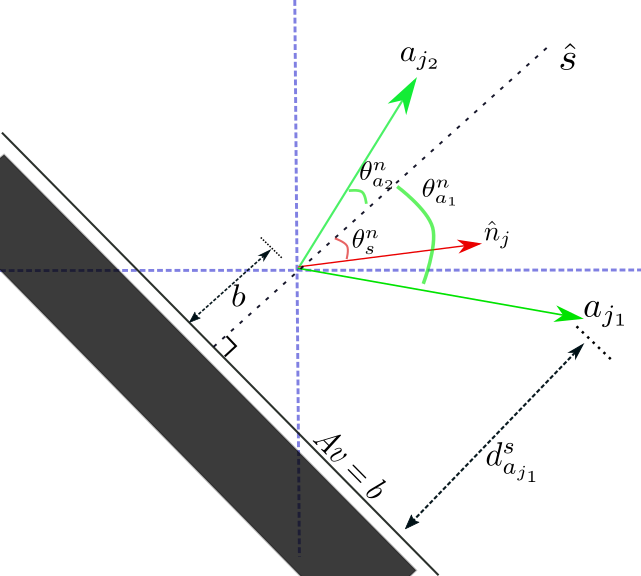}
    \caption{\small{Illustration of Half-space constraint. $Av\geq b$ represents the halfspace described by \eqref{eq::allowed_halfspace} with $A = \frac{\partial h_{ij}}{\partial x_j}$, $v=\dot{x}_j$, and $b =  -\alpha h - \max \left\{ \frac{\partial h_{ij}}{\partial x_i}\dot{x}_i \right\}$. $a_{j_1}, a_{j_2}$ are two instances of actual motions of agent $j$. $\hat{s}$ is the normal to the hyperplane $Av=b$. $\hat n_j$ is the nominal direction of motion for $j$.}
    }
    \label{fig:allowed_space}
\end{figure}

The margin by which CBF condition (\ref{eq::cbf_condition}) is satisfied i.e., the value of $\dot{h}_{ij}+\alpha_{ij}h_{ij}$,  and the best action agent $i$ can implement to increase this margin for a given motion $\dot{x}_j$ of agent $j$, are both important cues that help infer the nature of $j$. The adaptation of parameter $\alpha_{ij}$ by agent $i$ depends on the following:\\ 
% \begin{enumerate}
    % \item 
    \noindent \textbf{(1)Allowed motions of robot $\bm j$}: The worst a robot $j$ is allowed to perform in $i$'s perspective is the critical point beyond which $i$ cannot find a feasible solution to (\ref{eq::CBF_QP}). This is formulated as follows
    \eqn{
      \frac{\partial h_{ij}}{\partial x_j}\dot{x}_j &\geq -\alpha_{ij} -\frac{\partial h_{ij}}{\partial x_i}\dot{x}_i, \\
      &\geq -\alpha h - \max_{\dot{u}_i} \left\{ \frac{\partial h_{ij}}{\partial x_i}\dot{x}_i \right\}.
      \label{eq::allowed_halfspace}
    }
    This gives a lower bound that, if violated by $j$, will render $i$ unable to find a feasible solution to (\ref{eq::CBF_QP}). The required maximum value in Eq.(\ref{eq::allowed_halfspace}) can be obtained from the following Linear Program (LP)
    %\eqn{
     %   \max \quad & \frac{\partial h_{ij}}{\partial x_i}\dot{x}_i \\
        %\textrm{s.t.} \quad & \textrm{other barrier constraints hold} 
      %  \textrm{s.t.} \quad & \dot{h}_{ik} \geq -\alpha_{ik} h_{ik} %~~\forall k\in \mathcal{V}\setminus {i,j}
    %}
\begin{subequations}\label{eq::LP}
    \begin{align}
        \begin{split}\label{eq::LPa} 
             \max_{u_i} \quad & \frac{\partial h_{ij}}{\partial x_i}\dot{x}_i 
             \vspace{-0.1cm}
        \end{split}\\ 
        \begin{split} \label{eq::LPb}
            \textrm{s.t.} \quad & \dot{h}_{ik} \geq -\alpha_{ik} h_{ik}, ~~\forall k\in \mathcal{V}\setminus {i,j}.
        \end{split}
    \end{align}
% \vspace{-0.1cm}
\end{subequations}
    Equation \eqref{eq::allowed_halfspace} represents a half-space whose separating hyperplane has the normal direction $\hat s=\partial h_{ij}/\partial x_j$, and has been visualized in Fig.\ref{fig:allowed_space}.
    
    % \item 
\noindent    \textbf{(2)Actual motion of robot $\bm j$}: This is given by Eq.(\ref{eq::actual_motion}).
    
    % \item
\noindent    \textbf{(3)Nominal behavior of robot $\bm j$}: Suppose $i$ has knowledge of $j$'s target state. A nominal direction of motion $\hat{n}_j$ of $j$ in $i$'s perspective can be obtained by considering the Lyapunov function $V^j_i = ||x_j-x_j^r||^2$ and one can write 
    %Then the ideal direction of motion $\hat n_j$ for robot $j$ without regard to safety to any other robot and if it were an intact is given as
    \eqn{
        \hat{n}_j = -\frac{\partial V^j_i}{ \partial x_j }/|| \frac{\partial V^j_i}{ \partial x_j } ||.
    }
    This plays an important role in shaping belief as it helps to distinguish between uncooperative and adversarial agents. Note that $x_j^r$ need not be the true target state of agent $j$ for our algorithm to work. In fact, if we do not know $x_j^r$, we can always a assume a worst-case scenario where $x_j^r=x_i$, which corresponds to an adversary. With time though, our algorithm will learn that $j$ is not moving along $\dot{x}_j^r$ and increases its trust.
    
% \end{enumerate}
The two quantities of interest from Fig.\ref{fig:allowed_space} are the distance of $a_j^i$ from hyperplane, which is the margin by which CBF condition (\ref{eq::cbf_condition}) is satisfied, and the deviation of actual movement $a_j^i$ from the nominal direction $\hat n_j$.
%Both the direction of $a_j$ w.r.t $\hat n$ and the safety margin by which CBF inequality is satisfied are cues to form belief about $j$. These two can be designed independently as follows.

% \textbf{This section is very crucial but is not written very well. Some more elaboration on how the models are built here is needed.}

\subsubsection{Distance-based Trust Score}

Let the half-space in \eqref{eq::allowed_halfspace} be represented in the form $Av \geq b$ with $v = \dot{x}_j$ and $A,b$ defined accordingly. The distance $d_v^s$ of a vector $v$ from the dividing hyperplane (see Fig. \ref{fig:allowed_space}) is given by 
\eqn{
d_v^s = b - Av,
}
where $d_s^v<0$ is an incompatible vector, and is a scenario that should never happen. For $d_s^v>0$, its numerical value tells us by what margin CBF constraint is satisfied. Therefore, distance-based trust score is designed as
\eqn{
\rho_d = f_d(d^s_{a_j}),
\label{eq::rho_d}
}
where $f_d:\reals^+\rightarrow[0,1]$ is a monotonically-increasing, Lipschitz continuous function, such that $f_d(0)=0$. An example would be $f_d(d) = \tanh(\beta d)$, with $\beta$ being a scaling parameter.

\subsubsection{Direction-based Trust Score}

\noindent Suppose the angle between the vectors $\hat n$ and $\hat s$ is given by $\theta_s^n$ and between $a_j^i$ and $\hat s$ by $\theta_s^a$. The direction-based trust is designed as
% \eqn{
% \rho_\theta = \left\{
% 	\begin{array}{cl}
% 		1  & \mbox{if } \theta_s^a > \theta^n_s, \theta_a^n \geq \bar \theta  \\
% 		\beta_2 + \theta_a^n/\bar \theta & \mbox{if } \theta_s^a > \theta^n_s, \theta_a^n < \bar \theta \\
% 		\beta_2 & \mbox{if } \theta_s^a = \theta^n_s \\
% 		\beta_1 + (\beta_2-\beta_1)\theta_s^a/\theta_s^n & \mbox{if } \theta_s^a < \theta^n_s
% 	\end{array}
% \right.
% }
\eqn{
\rho_\theta = f_\theta( \theta^n_s/\theta^a_s ),
%\rho_\theta = \beta_2 + \tanh{(\beta_1\theta^n_s/\theta^a_s)}
\label{eq::rho_theta}
}
where $f_\theta:\reals^+ \rightarrow [0,1]$ is again a monotonically-increasing Lipschitz continuous function with $f_\theta(0)=0$. Note that even if $\theta^a_s=\theta^n_s$, i.e., $j$ is perfectly following its nominal direction, the trust may not be $1$ as the robot might be uncooperative. However, when $\theta^a_s<\theta^n_s$, as with $a_{j_2}$ in Fig. \ref{fig:allowed_space}, $j$ seems to be compromising its nominal movement direction for improved safety, thus leading to a higher score. Finally, when $\theta^a_s<\theta^n_s$, as with $a_{j_1}$ in Fig. \ref{fig:allowed_space}, then $j$ is doing worse for inter-robot safety than its nominal motion and is therefore either uncooperative/adversarial or under the influence of other robots, both of which lead to lower trust.

%$\beta_1,\beta_2$ are chosen such that $\rho_\theta\in [0,1]$ with higher values for a cooperative and cooperative agent.

%Here $\beta_1 \neq 0$ because it does not represent the worst case scenario. It represents the case when system is doing the least for safety. We would like 0 to represent the absolute worst case where the adversary is trying to intentionally collide with other robots. Similarly
%$\beta_1< 1$ because a robot going towards it's goal by exactly following $\hat n$ might be uncooperative and simply ignoring other robots.

\subsubsection{Final Trust Score}
The trust metric is now designed based on $\rho_d$ and $\rho_{\theta}$. Let $\bar \rho_d\in (0,1)$ be the desired minimum robustness in satisfying the CBF condition. Then, the trust metric $\rho\in [-1,1]$ is designed as follows:
\eqn{
\rho = \left\{ \begin{array}{cc}
    (\rho_d-\bar \rho_d)\rho_\theta,  & \mbox{if } \rho_d \geq \bar \rho_d,\\
     (\rho_d-\bar \rho_d)(1-\rho_\theta),  & \mbox{if } \rho_d < \bar \rho_d.
\end{array} \right.
\label{eq::trust}
}
Here, $\rho=1,-1$ represent the absolute belief in another agent being cooperative and adversary, respectively. If $\rho_d>\bar \rho_d$, then we would like to have $\rho>0$, and its magnitude is scaled by $\rho_\theta$ with smaller values of $\rho_\theta$ conveying low trust. Whereas, if $\rho_d-\bar \rho_d<0$, then we would like the trust factor to be negative. A smaller $\rho_\theta$ in this case implies more distrust and should make the magnitude larger, hence the term $1-\rho_\theta$.

% For a desired minimum robustness in satisfying the safety constraint,  When $\rho_d$ is close to 1, we would like to trust $j$ regardless of $\rho_\theta$ to prevent conservative behavior. Then for some thresholds $\bar \rho_m>0$, the trust is defined as
% \eqn{
%   \rho = \left\{
%         \begin{array}{cl}
%              \rho_\theta \rho_m &  \mbox{if } \rho_m > \bar \rho_m \\
%              \rho_\theta \rho_m &  \mbox{if } \rho_m < \bar \rho_m, \rho_\theta>0.5\\
%              -(1-\rho_\theta)(1-\rho_m) &  \mbox{if } \rho_m < \bar \rho_m, \rho_\theta<0.5
%         \end{array}
%   \right.
%   \label{eq::rho}
% }
% In this design $\rho_\theta$ decides the direction of change and $\rho\theta,\rho_m$ decides

The trust $\rho$ is now used to adapt $\alpha$ with following equation
\eqn{
    \dot{\alpha}_{ij} = f_{\alpha_{ij}}(\rho),
    \label{eq::alpha_dot}
}
where $f_{\alpha_{ij}}:[-1,1]\rightarrow \reals$
%, f_\alpha(0)=0$
 is a monotonically increasing function. A positive value of $f_{\alpha_{ij}}$ relaxes the CBF condition by increasing $\alpha$, and a negative value decreases $\alpha$. The framework that each agent $i$ implements at every time $t$ is detailed in Algorithm. \ref{algo::framework}.

\begin{algorithm}[ht]
\caption{Trust-based Multi-agent CBF}

\begin{algorithmic}[1]
    \Require $x, i, f_\alpha, f_\theta, f', \bar \rho_d$ 
    \State $\alpha_{ij} \leftarrow \text{Current value of CBF Parameters}$ 
    %  $ \Comment{slice w.r.t parameter $k$}
    \For{all time}
    \For{$j \in \mathcal{V}\setminus i$}
        \State Predict $a_j^i$ \Comment{Actual movement of $j$}
        \State Compute $\rho_\theta,\rho_d$ with Eqns.\eqref{eq::rho_theta},\eqref{eq::rho_d}
        \State Compute $\rho$ with \eqref{eq::trust}\Comment{Trust score}
        \State Calculate nominal input using \eqref{eq::CLF}.
        \State Implement $u_i$ from \eqref{eq::CBF}.
        \State Update $\alpha_{ij}$ with update rule (\ref{eq::alpha_dot})
    \EndFor
    \EndFor
\end{algorithmic}

\label{algo::framework}
\end{algorithm}

\begin{remark}
Note that equation \eqref{eq::trust} does not represent a locally Lipschitz continuous function of $\rho_\theta,\,\rho_d$. This poses theoretical issues as having a non-Lipschitz gradient flow in \eqref{eq::alpha_dot} would warrant further analysis concerning on existence of unique solution for the combined dynamical involving $x$ and $\alpha$. Hence, while \eqref{eq::trust} represents our desired characteristics, we can use a sigmoid function to switch across the boundary thereby holding the assumptions in Theorem \ref{theorem::tunable_CBF} true.
\end{remark}

% \section{Allowed Values of $\dot{\alpha}$}
% \red{Hardik: this subsection needs to be completed}
% The CBF constraint $d_s>0$ needs to be satisfied at the next time step so that $d_s + \partial d_s >0$. This can be solved to have a bound on $\partial \alpha$ as
% \eqn{
%   \partial \alpha \geq -(d_s + \partial d_s|_\alpha)/h
% }
% This condition ensures that a solution to CBF equation exists at next time step. The gradient designed before can be written as
% \eqn{
% \dot{\alpha} = f(d_s,\theta^a_s,\theta^n_s)
% }

\subsection{Sufficient Conditions for a Valid trust Metric}
In the presence of multiple constraints, feasibility of the QP (\ref{eq::CBF_QP}) might not be guaranteed even if the control input is unbounded and the states are far away from boundary, i.e., $h_{ij}>>0$. In this section, we derive conditions that guarantee that the QP does not become infeasible.
%In the following subsection, given that at initial time there exists a solution to (\ref{eq::CBF_QP}) and Assumption \ref{assumption::estimate} holds, we present conditions on $\dot{\alpha}$ that guarantee existence of a trajectory to augmented system (\ref{eq::dynamics}),(\ref{eq::alpha_dot}) $\forall t>0$ until the boundary condition $h_{ij}=0$ is reached.

\begin{assumption}
For a feasible solution of QP in \eqref{eq::theorem_multi_barriers}, the resulting $u(x,\alpha) \in \U$ is always Lipschitz continuous in $x$ and $\alpha_j$. Note that, in order to hold Assumption \ref{assumption::true_estimate}, one needs Assumption \ref{assumption::control_Lipschtiz} to be satisfied.
\label{assumption::control_Lipschtiz}
\end{assumption}

\begin{assumption}
Suppose the initial state satisfies the barrier constraints strictly, i.e., $h_{ij}(x_i(0),x_j(0)) > 0, ~j\in \mathcal{V}\setminus i$. Now suppose there always exist $\alpha_{ij}>0$ at time $t$ such that the QP (\ref{eq::CBF_QP}) is feasible, then $h_{ij}(x_i(t),x_j(t)) > 0~ \forall t>0$ .
\label{assumption::safety_possible}
\end{assumption}
%This assumption allows us to deal with cases where infeasibility of (\ref{eq::CBF_QP}) stems only from conflicting CBF constraints.

\begin{theorem}(Compatibility of Multiple CBF Constraints)
Under Assumptions \ref{assumption::true_estimate}-\ref{assumption::safety_possible}, consider the system dynamics in  \eqref{eq::dynamics_general} subject to multiple barrier functions $h_j, j \in \mathcal{V}_M =  {1,2..,M}$.
 Let the controller $u(x,\alpha_j)$ be formulated as
\begin{subequations}\label{eq::theorem_multi_barriers}
    \begin{align}
        \begin{split}\label{eq::theorem_multi_barriersa} 
         u(x,\alpha_j) = \arg \min_{u} \quad & || u - u^r ||^2
         \vspace{-0.1cm}
        \end{split}\\ 
        \begin{split} \label{eq::theorem_multi_barriersb}
        \textrm{s.t.} \quad & \dot{h}_{j}(x,u) \geq -\alpha_j h_j(x), \forall j\in \mathcal{V}_M,
        \end{split}
    \end{align}
% \vspace{-0.1cm}
\label{eq::th2_eq}
\end{subequations}
with corresponding parameters $\alpha_j$. Now suppose $\alpha_j$, as an augmented state, has following dynamics
\eqn{
   \dot \alpha_j = f_{j}(x,\alpha_1,\alpha_2,..,\alpha_M). 
   \label{eq::alphaj_dynamics}
}
Let $L_{h_j}$ and $L_{\dot{h}_j}$ be the Lipschitz constants of $h_j(x)$ and $\dot{h}_j(x,\dot{x})$. Suppose a solution to (\ref{eq::th2_eq}) exists at initial time $t=0$ and $\dot{\alpha}_j$ satisfies the following condition with $B(x)=||\hat{F}(x)|| + b_F(x)$ (see Assumption \ref{assumption::estimate})
\eqn{
 \dot{\alpha}_j \geq \frac{ -(  d_j + L_{\dot h_{j}}L_F B(x)^2  + \alpha_j L_{h_j}B(x)   ) }{h_j}.
  \label{eq::alpha_cond}
}
Then, a solution to (\ref{eq::th2_eq}) continues to exist until $h_j= 0, j\in \mathcal{V}_M$.  
%\eqref{eq::theorem_multi_barriers} if $h_j\neq 0$.
\end{theorem}

\begin{proof}
Based on Assumptions \ref{assumption::true_estimate} and \ref{assumption::estimate},  one can write $\dot{x}\leq F(x)$ and $\delta \dot{x}\leq L_F || \delta x ||$, where $\delta y$ denotes infinitesimal variation in $y$. In order to have a feasible solution for QP in \eqref{eq::theorem_multi_barriers},  one needs to satisfy the following 
\begin{equation}
  d_j(x,\dot{x}) = \dot{h}_j(x,\dot{x}) + \alpha_j h_j(x) \geq 0. 
\end{equation}
%As we know, the whole equation is dependent on $x$ and $\dot{x}$, 
The variation in $d_j(x,\dot{x})$ is given by
\eqn{
  \delta d_j = \delta \dot{h}_j + h_j \delta \alpha_j  + \alpha_j \delta h_j.
}
For a feasible solution to QP in \eqref{eq::theorem_multi_barriers} with variation in  $\delta d_j$, one needs  $d_j + \partial d_j\geq 0 $. Thus,  the following condition holds
\eqn{
    \delta \alpha_j \geq \frac{-(d_j + \delta \dot{h}_j  + \alpha_j \delta h_j )  }{h_j}. 
    \label{eq::alpha_bd1}
}
In multi-agent setting, $h_j$ depends only a subset of $x$, namely $x_{j_1},x_{j_2}$ where $j_1,j_2\in \{1,2,\ldots,N\}$ of the whole state vector $x=[x_1,\ldots,x_N]$. However, the evolution of state dynamics $\dot{x}$ depends on control input $u$ computed based on QP formulation in \eqref{eq::theorem_multi_barriers}, $u$ depends on all the constraints $h_j$ and states $x_j$. Thus, the $\delta h_j$ and $\delta \dot x_j$ are dependent on the state of all agents and not only $x_{j_1},x_{j_2}$. This introduces coupling between all the constraints which is hard to solve for. However, given that $\dot{x}=F(x)$ with known Lipschitz constant $L_F$, we can decouple all constraints. Using $ ||\delta \dot{x}|| \leq~L_F ||\delta x|| , |\delta h_j(x)| \leq~L_{h_j} ||\delta x||$, and 
\eqn{
  |\delta \dot{h}_j(x,\dot{x})| \leq L_{\dot{h}_j} \bigg|\bigg| \begin{bmatrix} \delta x \\ \delta \dot{x}  \end{bmatrix} \bigg|\bigg| \leq L_{\dot{h}_j} || \delta x || ~ || \delta \dot{x} || \nonumber,
}
 we can write (\ref{eq::alpha_bd1}) as

% \eqn{
% \partial x \leq F(x) \\
% ||\partial \dot{x}|| \leq L_F ||\partial x|| \\
% |\partial h_j(x)| \leq L_{h_j} ||\partial x|| \\
% |\partial \dot{h}_j(x,\dot{x})| \leq L_{\dot{h}_j} \bigg|\bigg| \begin{bmatrix} \partial x \\ \partial \dot{x}  \end{bmatrix} \bigg|\bigg| \leq L_{\dot{h}_j} || \partial x || ~ || \partial \dot{x} ||
% }
% we can get 
\eqn{
  \delta \alpha_j \geq \frac{ -(  d_j + L_{\dot h_{j}}L_F||\delta x||^2  + \alpha_j L_{h_j}||\delta x||   ) }{h_j}.
}
Finally, the change $\delta x$ caused by $\dot \alpha$ is given by $\dot{x}=F(x)$ which . From Assumption \ref{assumption::estimate}, we further have a bound available on $||F(x)||$ which is $B(x)=||\hat{F}(x)|| + b_F(x)$. Therefore, if $\dot \alpha$ satisfies (\ref{eq::alpha_cond}), we have that $d_j + \delta d_j\geq 0, ~ \forall j$ and hence a solution for QP in \eqref{eq::theorem_multi_barriers} exists. This completes the proof. 
%Note that the above proof does not account for control input bounds and will be treated in future. In fact, in presence of control input bounds, the solution of QP with single barrier function is also not guaranteed to be Lipschtiz continuous and an analysis based on non-smooth control theory may be applied[].

\end{proof}
\begin{remark}
Note that when Theorem 2 is applied to design control input $u_i$ for an agent $i$ in the multi-agent setting with the concatenated state vector $x=[x_1,\ldots,x_N]$, where $x_i$ evolves based on $u_i(x, \alpha_{ij})$, the remaining states are assumed to have closed-loop Lipschtiz continuous dynamics with known bounds that is used to determine $F(x)$ in Theorem 2 from $F_j(x)$ in Assumption \ref{assumption::estimate}. 
\end{remark}

\begin{remark}
Also, consider a boundary condition scenario, i.e.,  when $h_j\rightarrow 0$ for which $\alpha_j\rightarrow \infty$. Such scenario can happen, for example, when an agent is surrounded by adversaries and has no escape path. In order to keep the QP feasible, the agent will have to keep increasing $\alpha_j$ until $h_j=0$ and then collision is unavoidable. We would expect that if there were an escape direction, then atleast one of $\dot{\alpha}_j$ would have lower bound $\leq 0$ before $h_j= 0$ is reached. Such inferences will be formally addressed in future work.
\end{remark}

\section{Simulation Results}
The proposed Algorithm 1 is implemented to design controllers for waypoint navigation of a group of three intact robots while satisfying desired constraints. The intact agents are modeled as unicycles with states given by the position coordinates $p_x,p_y$ and the heading angle $\psi$ w.r.t. a global reference frame, $\dot{p}_x = v\cos\theta, ~\dot{p}_y =v\sin\theta,~ \dot{\psi} = \omega$, where $v,\omega$ are linear and angular velocity expressed in the body-fixed frame, and act as control inputs. The adversarial and uncooperative agents are modeled as single integrators with dynamics $\dot{p}_x = v_x ,~ \dot{p}_y = v_y$, where $v_x,v_y$ are velocity control inputs. Fig.\ref{fig::scenario2} shows a scenario where an adversarial robot chases one of the intact agents, and two other uncooperative agents move along horizontal paths without regard to any other agent. The nominal trajectories in this case are straight lines from initial to target locations. None of the intact robots know the identity of any other robot in the system, and initialize $\alpha_{ij}$ to 0.8 uniformly. It can be seen that intact agents are successfully able to remain close to the nominal paths and reach their target location in given time. The fixed $\alpha$ case, on the other hand, fails to reach the goal and diverges away from nominal paths. %\st{\textbf{UNCLEAR: REPHRASE} Note that the total distance traveled by each agent in this case is also much smaller than the proposed algorithms.} 
Fig.\ref{fig::trust},\,\ref{fig::alphas}, and \ref{fig::barriers} illustrate the variation of trust metric, $\alpha_{ij}s$, and CBFs with time. The adversarial agent uses an exponentially stabilizing CLF to chase agent 1. The reference velocity for unicycles using was computed with a controller of the form: $v = k_s e_s , \omega = k_\omega (\arctan (e_y/e_x) - \psi )$, where $k_s,k_\omega>0$ are gains, both taken as $2.0$ in simulations, $e_s$ is the distance to the target location, and $e_x,e_y$ are {position errors of the agent to its nominal trajectory in X,Y coordinates} respectively. The trust metric was computed using (\ref{eq::trust}) with $\bar\rho_d=1.0$. We use first-order barrier functions for {unicycles model \cite{wu2016safety}} to enforce collision avoidance, and $f_d=\tanh(d),f_\theta=\tanh(2\theta^n_a/\theta^2_s)$ in Eqns.(\ref{eq::rho_d}),(\ref{eq::rho_theta}). The simulation video and code can be found at \textit{https://github.com/hardikparwana/Adversary-CBF}.

Assumption \ref{assumption::estimate} was realized by assuming %\textbf{UNCLEAR: DID YOU VERIFY AFTER THE SIMULATION? YOU CAN'T CHECK THE CONDITION A PRIORI?} 
(and verifying through observations) that the maximum normed difference between $\dot{x}_i(t)$, which is to be predicted, and $\dot{x}_i(t-\Delta t)$, is 10\% of the norm of $\dot{x}_i(t-\Delta t)$. This is reasonable as the simulation time step $\Delta t$ is 0.05 sec, and because the adversaries and intact agents all use Lipschitz continuous controllers. Note that the case with fixed $\alpha$ in Fig.\ref{fig::scenario2} uses the same assumption to design a controller, but still fails to reach the goal. 

\begin{figure}[h!]
    \centering
    \includegraphics[scale=0.4]{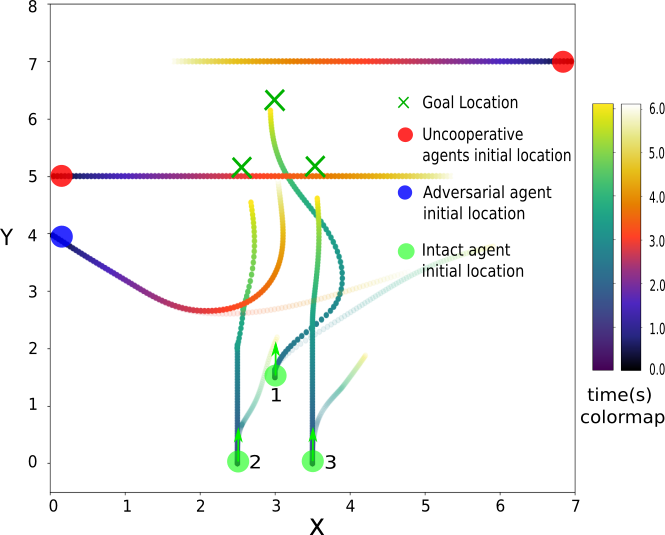}
    \caption{\small{Intact agents navigating through an environment with non-cooperative agents. The timestamp of different points on the trajectory is given by the colormap. The bold colors show the path resulting from the proposed method. The paths with increased transparency result from application of CBFs with fixed $\alpha$ in (\ref{eq::CBF_QP}).}}
    \label{fig::scenario2}
    \vspace{-5mm}
\end{figure}

\begin{figure}[h!]
    \centering
    \includegraphics[scale=0.4]{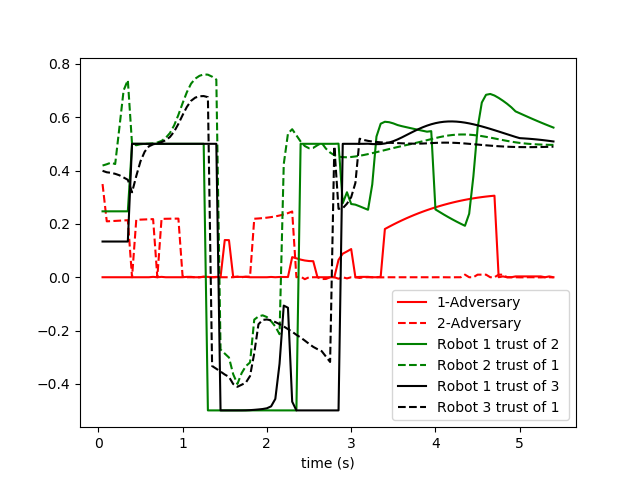}
    \caption{\small{Trust between different pairs of robots. Note that two agents need not have the same level of trust for each other. The initial trust for adversary is positive. This is because the initial value of $\alpha=0.8$ is already a very conservative value and trust based relation allows agent 1 to not deviate too much from nominal trajectory but still avoids collision.}}
    \label{fig::trust}
    \vspace{-5mm}
\end{figure}

% \begin{figure}[h!]
%     \centering
%     \includegraphics[scale=0.4]{cdc2022figures/barriers_robot_1.eps}
%     \caption{Variation of barrier functions of Robot 1 with time. The safe set boundary is h=0. Trust-based relaxation allows agents to go closer to the boundaries compared to fixed $\alpha$ case and hence leads to less conservative response while guaranteeing safety.}
%     \label{fig::barriers}
%     \vspace{-5mm}
% \end{figure}

% \begin{figure}[h!]
%     \centering
%     \includegraphics[scale=0.19]{cdc2022figures/alpha_cbf.png}
%     \caption{\small{Variation of barrier functions and $\alpha_{ij}$ of Robot 1 with time. The safe set boundary is h=0. Trust-based relaxation allows agents to go closer to the boundaries compared to fixed $\alpha$ case and hence leads to less conservative response while still guaranteeing safety.}}
%     \label{fig::barriers}
%     \vspace{-5mm}
% \end{figure}

\begin{figure}[h!]
    \centering
    \includegraphics[scale=0.3]{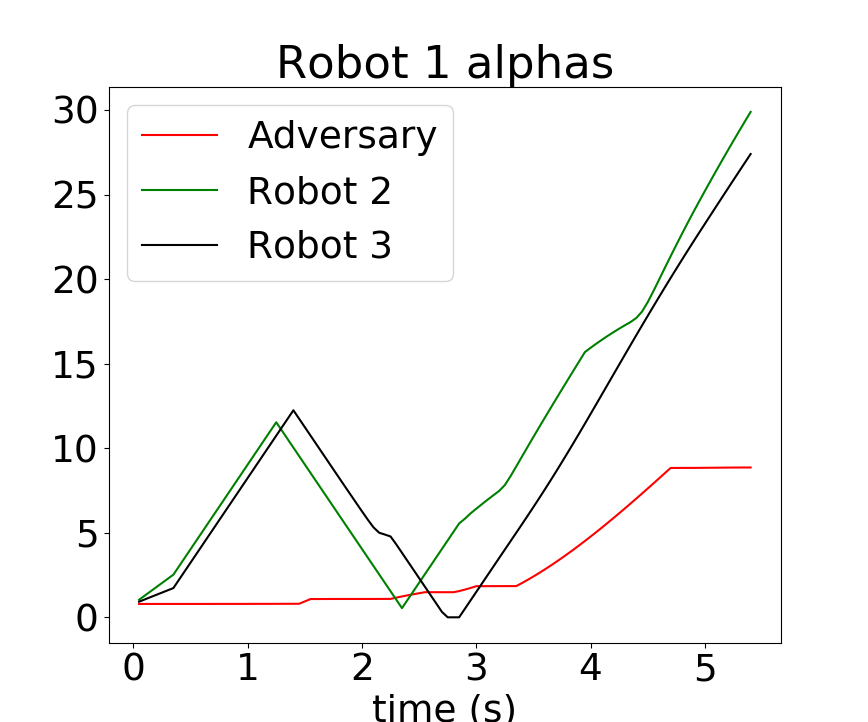}
    \caption{\small{Variation of $\alpha_{1j}$ of Robot 1 with Robot 2,3 and the adversarial agent. Trust-based adaptation allows $\alpha$ to increase, thus relaxing the constraints.}}
    \label{fig::alphas}
    \vspace{-5mm}
\end{figure}

\begin{figure}[h!]
    \centering
    \includegraphics[scale=0.25]{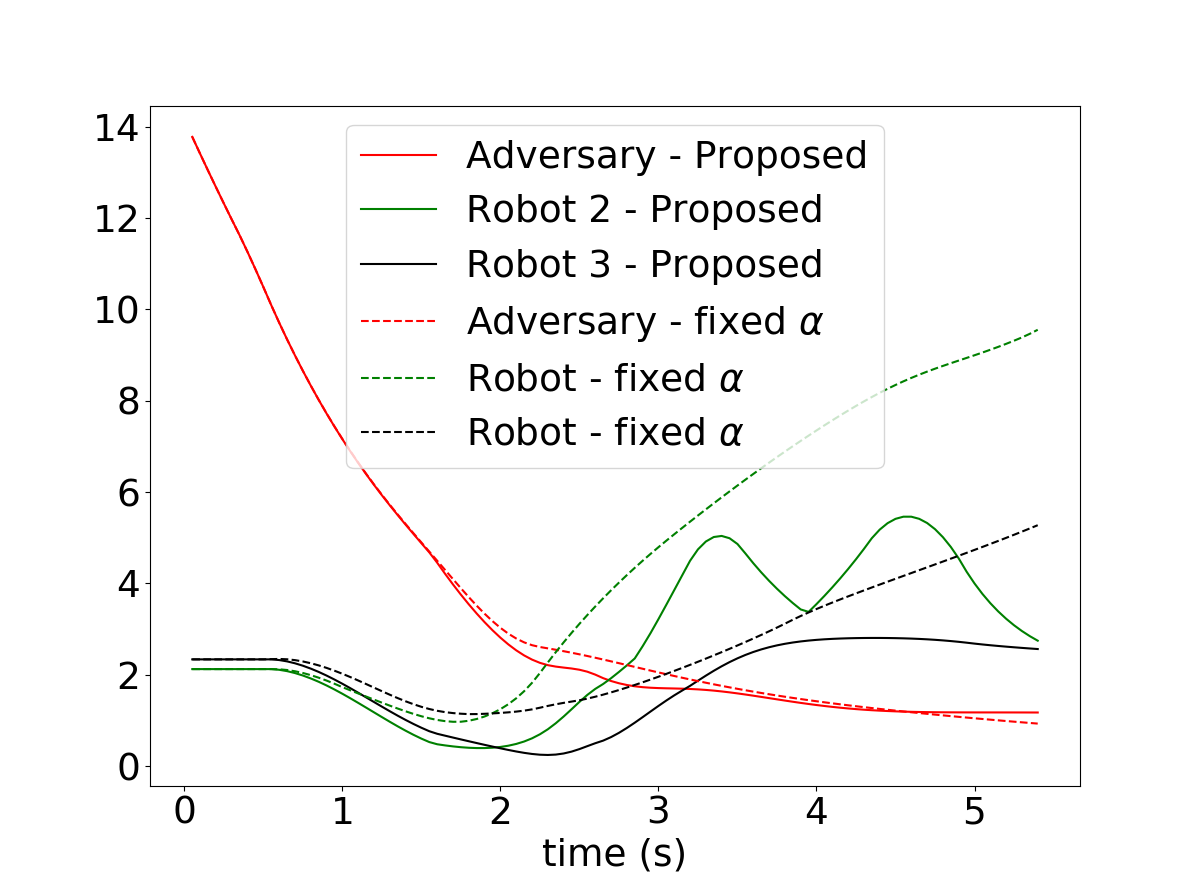}
    \caption{\small{Variation of barrier functions of Robot 1 with time. The plots represent the value of barrier function used by agent 1 for collision avoidance with agent 2,3 and the adversarial agent. The barriers with uncooperative agents are not shown but can be found along with our videos. The safe set boundary is h=0. Trust-based relaxation allows agents to go closer to the boundaries compared to fixed $\alpha$ case and hence leads to less conservative response while still guaranteeing safety.}}
    \label{fig::barriers}
    \vspace{-5mm}
\end{figure}

\subsection{Conclusion and Future Work}
This paper introduces the notion of trust for multi-agent systems where the identity of robots is unknown. The trust metric is based on the robustness of satisfaction of CBF constraints. It also provides a direct feedback to the low-level controller and help shape a less conservative response while ensuring safety. The effect of input constraints and the sensitivity of the algorithm to its parameters and function choices will be evaluated in future work.
%wrt to sufficient conditions for a valid trust metric
\bibliographystyle{IEEEtran}
\bibliography{cdc.bib}

\end{document}